\documentclass[11pt]{article}
\usepackage{amsmath,amssymb,amsfonts,amsthm}
\usepackage{float}
\numberwithin{equation}{section}
\newtheorem{theorem}{Theorem}[section]
\newtheorem{definition}{Definition}[section]

\begin{document}
\begin{center}
{\Large{\textbf{A Note on the Gutman Index of Jaco Graphs}}} 
\end{center}
\vspace{0.5cm}
\centerline{\large{Johan Kok}} 
\centerline{\small{Tshwane Metropolitan Police Department}}
\centerline{\small{City of Tshwane, Republic of South Africa}}
\centerline{\tt {kokkiek2@tshwane.gov.za}}
\vspace{0.5cm}
\centerline{\large{Susanth C}} 
\centerline{\small{Department of Mathematics}}
\centerline{\small{Vidya Academy of Science and Technology}}
\centerline{\small{Thalakkottukara, Thrissur-680501, India}}
\centerline{\tt {susanth\_c@yahoo.com}}
\vspace{0.5cm}
\centerline{\large{Sunny Joseph Kalayathankal}} 
\centerline{\small{Department of Mathematics}}
\centerline{\small{Kuriakose Elias College}}
\centerline{\small{Mannaman, Kottayam-686561, India}}
\centerline{\tt {sunnyjoseph2014@yahoo.com}}
\vspace{0.5cm}
\begin{abstract}
\noindent The concept of the \emph{Gutman index,} denoted $Gut(G)$ was introduced for a connected undirected graph $G$. In this note we apply the concept to the underlying graphs of the family of Jaco graphs, (\emph{directed graphs by definition}), and decribe a recursive formula for the \emph{Gutman index} $Gut(J^*_{n+1}(x))$. We also determine the \emph{Gutman index} for the trivial \emph{edge-joint} between Jaco graphs.
\end{abstract}
\noindent {\footnotesize \textbf{AMS Classification Numbers:} 05C12, 05C20, 05C38, 05C40, 05C75} 
\section{Introduction}
\noindent For general reference to notation and concepts of graph theory see [2]. Unless mentioned otherwise, a graph $G = G(V,E)$ on $\nu(G)$ vertices (order of $G$) with $\epsilon(G)$ edges (size of $G$) will be a finite undirected and connected simple graph. The degree of a vertex in $G$ is denoted $d_G(v)$ and if the context of $G$ is clear the degree is denoted $d(v)$ for brevity. Also in a directed graph $G^\rightarrow$ the degree is $d_{G^\rightarrow}(v) = d^+_{G^\rightarrow}(v) +  d^-_{G^\rightarrow}(v)$ or for brevity, $d(v) = d^+(v) +  d^-(v)$ if $G$ is clear. \\\\
The concept of the Gutman index $Gut(G)$ of a connected undirected graph $G$ was introduced in 1994 by Gutman [4]. It is defined to be $Gut(G) = \sum\limits_{\{v, u\} \subseteq V(G)} d_G(v)d_G(u)d_G(v, u),$ where $d_G(v)$ and $d_G(u)$ are the degree of $v$ and $u$ in $G$ respectively, and $d_G(v, u)$ is the distance between $v$ and $u$ in $G$. Clearly, if the vertices of $G$ of order $n$ are randomly labeled $v_1, v_2, v_3, ..., v_n$ the definition states that $Gut(G) = \sum\limits_{\ell=1}^{n-1}\sum\limits_{j=\ell+1}^{n}d_G(v_\ell)d_G(v_j)d_G(v_\ell, v_j).$ Worthy results are reported in Andova et al. [1] and Dankelmann et al. [3].
\section{The Gutman Index of the Underlying Graph of a Jaco Graph}
Despite earlier definitions in respect of the family of Jaco graphs [5, 6], the definitions found in [7] serve as the unifying definitions. For ease of reference some of the important definitions are repeated here.
\begin{definition}$[7]$
Let $f(x) = mx + c; x \in \Bbb N,$ $m,c\in \Bbb N_0.$ The family of infinite linear Jaco graphs denoted by $\{J_\infty(f(x)):f(x) = mx + c; x \in \Bbb N$ and $m,c \in \Bbb N_0\}$ is defined by $V(J_\infty(f(x))) = \{v_i: i \in \Bbb N\}$, $A(J_\infty(f(x))) \subseteq \{(v_i, v_j): i, j \in \Bbb N, i< j\}$ and $(v_i,v_ j) \in A(J_\infty(f(x)))$ if and only if $(f(i) + i) - d^-(v_i) \geq j.$
\end{definition}
\begin{definition}$[7]$
The family of finite linear Jaco graphs denoted by $\{J_n(f(x)):f(x) = mx + c; x \in \Bbb N$ and $m,c\in \Bbb N_0\}$ is defined by $V(J_n(f(x))) = \{v_i: i \in \Bbb N, i \leq n \}$, $A(J_n(f(x))) \subseteq \{(v_i, v_j): i,j \in \Bbb N, i< j \leq n\}$ and $(v_i,v_ j) \in A(J_n(f(x)))$ if and only if $(f(i) + i) - d^-(v_i) \geq j.$
\end{definition} 
The reader is referred to [7] for the definition of the \emph{prime Jaconian vertex} and the \emph{Hope graph}. The graph has four fundamental properties which are:\\
(i) $V(J_\infty(f(x))) = \{v_i:i \in \Bbb N\}$ and,\\
(ii) if $v_j$ is the head of an arc then the tail is always a vertex $v_i,$ $i<j$ and,\\
(iii) if $v_k,$ for smallest $k \in \Bbb N$ is a tail vertex then all vertices $v_ \ell,$ $k< \ell<j$ are tails of arcs to $v_j$ and finally,\\
(iv) the degree of vertex $k$ is $d(v_k) = f(k).$\\ \\
The family of finite directed graphs are those limited to $n \in \Bbb N$ vertices by lobbing off all vertices (and arcs to vertices) $v_t, t > n.$ Hence, trivially $d(v_i) \leq i$ for $i \in \Bbb N.$  For $m=0$ and $c \geq 0$ two special classes of disconnected linear Jaco graphs exist. For $c= 0$ the Jaco graph $J_n(0)$  is a null graph (\emph{edgeless graph}) on $n$ vertices. For $c> 0$, the Jaco graph $J_n(c) = \bigcup\limits_{\lfloor\frac{n}{c+1}\rfloor-copies}K^\rightarrow_{c+1} \bigcup K^\rightarrow_{n-(c+1)\cdot \lfloor\frac{n}{c+1}\rfloor}.$ since the Gutman index is defined for connected graphs the bound $m\geq 1$ will apply.\\\\
In this note we only consider the case $m = 1$, $c=0.$ The generalisation for $f(x) = mx + c$ in general remains open. Denote the underlying Jaco graph by $J^*_n(f(x))$. A recursive formula of the Gutman index $Gut(J^*_{n+1}(x))$ in terms of $Gut(J_n^*(x))$ is given in the next theorem.
\begin{theorem}
For the underlying graph $J^*_n(x)$ of a finitie Jaco Graph $J_n(x), n \in \Bbb N, n \geq 2$ with Jaconian vertex $v_i$ we have that recursively:\\ \\
$Gut(J^*_{n+1}(x)) = Gut(J^*_n(x)) +\sum\limits_{k=1}^{i}\sum\limits_{t= i+1}^{n}d_{J^*_n(x)}(v_k)d_{J^*_n(x)}(v_k, v_t) + \sum\limits_{t=i+1}^{n-1}\sum\limits_{q = t+1}^{n}(d_{J^*_n(x)}(v_t) + d_{J^*_n(x)}(v_q)) + (n-i)(\sum\limits_{k=1}^{i}d_{J^*_n(x)}(v_k)d_{J^*_n(x)}(v_k, v_n) + \sum\limits_{t=i+1}^{n}d_{J^*_n(x)}(v_t)) + (n- i-1) + i(n-i).$
\end{theorem}
\begin{proof}
Consider the underlying Jaco graph, $J^*_n(x), n \in \Bbb N, n \geq 2$ with prime Jaconian vertex $v_i$. Now consider $J^*_{n+1}(x).$ From the definition of a Jaco graph the extension from $J^*_n(x)$ to $J^*_{n+1}(x)$ adds the vertex $v_{n+1}$ and the edges $v_{i+1}v_{n+1}, v_{i+2}v_{n+1}, ..., v_nv_{n+1}.$\\ \\
\noindent Step 1: Consider any ordered pair of vertices $(v_k, v_q)_{k < q}, 1 \leq k \leq i-1,$ and $k+1 \leq q \leq i.$ By applying the definition of the Gutman index to this pair of vertices we have the term:\\ \\
$d_{J^*_{n+1}(x)}(v_k)d_{J^*_{n+1}(1)}(v_q)d_{J^*_{n+1}(x)}(v_k, v_q) = d_{J^*_n(x)}(v_k)d_{J^*_n(x)}(v_q)d_{J^*_n(x)}(v_k, v_q).$\\ \\
By applying this step $\forall v_k,1 \leq k \leq i-1,$ and $\forall v_q, k+1 \leq q \leq i$ with $k < q$ we obtain:\\ \\
$\sum\limits_{k=1}^{i-1}\sum\limits_{q=k+1}^{i}d_{J^*_n(x)}(v_k)d_{J^*_n(x)}(v_q)d_{J^*_n(x)}(v_k, v_q).$\\ \\
\noindent Step 2: Consider any vertex $v_k, 1\leq k \leq i$ and any other vertex $v_t, i+1 \leq t \leq n.$ By applying the definition of the Gutman index to this pair of vertices we have the term:\\ \\
$d_{J^*_{n+1}(x)}(v_k)d_{J^*_{n+1}(x)}(v_t)d_{J^*_{n+1}(x)}(v_k, v_t) = d_{J^*_n(x)}(v_k)(d_{J^*_n(x)}(v_t) + 1)d_{J^*_n(x)}(v_k, v_t) =\\ \\
d_{J^*_n(x)}(v_k)d_{J^*_n(x)}(v_t)d_{J^*_n(x)}(v_k, v_t) + d_{J^*_n(x)}(v_k)d_{J^*_n(x)}(v_k, v_t).$\\ \\
By applying this step $\forall v_k, 1 \leq k \leq i$ and $\forall v_t, i+1 \leq t \leq n$, we obtain:\\ \\ $\sum\limits_{k=1}^{i}\sum\limits_{t=i+1}^{n}d_{J^*_n(x)}(v_k)d_{J^*_n(x)}(v_t)d_{J^*_n(x)}(v_k, v_t) + \sum\limits_{k=1}^{i}\sum\limits_{t=i+1}^{n}d_{J^*_n(x)}(v_k)d_{J^*_n(x)}(v_k, v_t).$\\ \\
\noindent Step 3: Consider any two distinct vertices $v_t, v_q, i+1 \leq t \leq n-1,$ and $t+1 \leq q \leq n.$ By applying the definition of the Gutman index to this pair of vertices we have the term:\\ \\
$d_{J^*_{n+1}(x)}(v_t)d_{J^*_{n+1}(x)}(v_q)d_{J^*_{n+1}(x)}(v_t, v_q) = (d_{J^*_n(x)}(v_t)+1)(d_{J^*_n(x)}(v_q) + 1)d_{J^*_n(x)}(v_t, v_q) =\\ \\
d_{J^*_n(x)}(v_t)d_{J^*_n(x)}(v_q) + d_{J^*_n(x)}(v_t) + d_{J^*_n(x)}(v_q) + 1.$\\ \\
By applying this step $\forall v_t, i+1 \leq t \leq n-1$ and $\forall v_q, t+1 \leq q \leq n$, we obtain:\\ \\
$\sum\limits_{t=i+1}^{n-1}\sum\limits_{q = t+1}^{n}d_{J^*_n(x)}(v_t)d_{J^*_n(x)}(v_q) + \sum\limits_{t=i+1}^{n-1}\sum\limits_{q = t+1}^{n}(d_{J^*_n(x)}(v_t) + d_{J^*_n(x)}(v_q)) + (n- i-1).$\\ \\
\noindent Step 4: Consider any vertex $v_k, 1\leq k \leq i$ and the vertex $v_{n+1}.$ By applying the definition of the Gutman index to this pair of vertices we have the term:\\ \\
$d_{J^*_{n+1}(x)}(v_k)d_{J^*_{n+1}(x)}(v_{n+1})d_{J^*_{n+1}(x)}(v_k, v_{n+1}) = d_{J^*_n(x)}(v_k)(n-i)(d_{J^*_n(x)}(v_k, v_n) + 1).$ \\ \\
By applying this step $\forall v_k, 1 \leq k \leq i$ we obtain:\\ \\
$\sum\limits_{k=1}^{i}d_{J^*_n(x)}(v_k)(n-i)(d_{J^*_n(x)}(v_k, v_n) + 1) = (n-i)\sum\limits_{k=1}^{i}d_{J^*_n(x)}(v_k)d_{J^*_n(x)}(v_k, v_n) + i(n-i).$\\ \\
\noindent Step 5: Consider any vertex $v_t, i+1\leq t \leq n$ and the vertex $v_{n+1}.$ By applying the definition of the Gutman index to this pair of vertices we have the term:\\ \\
$d_{J^*_{n+1}(x)}(v_t)d_{J^*_{n+1}(x)}(v_{n+1})d_{J^*_{n+1}(x)}(v_t, v_{n+1}) = d_{J^*_n(x)}(v_t)(n-i)d_{J^*_n(x)}(v_t, v_n).$ \\ \\
By applying this step $\forall v_t, i+1 \leq t \leq n$ we obtain:\\ \\
$\sum\limits_{t=i+1}^{n}d_{J^*_n(x)}(v_t)(n-i) = (n-i)\sum\limits_{t=i+1}^{n}d_{J^*_n(x)}(v_t).$\\ \\ \\
\noindent \textbf{Final summation step:} Adding Steps 1 to 5 and noting that:\\ \\
$Gut(J^*_n(x)) = \sum\limits_{k=1}^{i-1}\sum\limits_{q=k+1}^{i}d_{J^*_n(x)}(v_k)d_{J^*_n(x)}(v_q)d_{J^*_n(x)}(v_k, v_q) + \sum\limits_{k=1}^{i}\sum\limits_{t=i+1}^{n}d_{J^*_n(x)}(v_k)d_{J^*_n(x)}(v_t)d_{J^*_n(x)}(v_k, v_t) +\\ \\ \sum\limits_{t=i+1}^{n-1}\sum\limits_{q = t+1}^{n}d_{J^*_n(x)}(v_t)d_{J^*_n(x)}(v_q),$ provides the result:\\ \\
$Gut(J^*_{n+1}(x)) = Gut(J^*_n(x)) +\sum\limits_{k=1}^{i}\sum\limits_{t=i+1}^{n}d_{J^*_n(x)}(v_k)d_{J^*_n(x)}(v_k, v_t) + \sum\limits_{t=i+1}^{n-1}\sum\limits_{q = t+1}^{n}(d_{J^*_n(x)}(v_t) + d_{J^*_n(x)}(v_q)) + (n-i)(\sum\limits_{k=1}^{i}d_{J^*_n(x)}(v_k)d_{J^*_n(x)}(v_k, v_n) + \sum\limits_{t=i+1}^{n}d_{J^*_n(x)}(v_t)) + (n- i-1) + i(n-i).$
\end{proof}
\section{The Gutman Index of the Edge-joint between $J^*_n(x), n \in \Bbb N$ and $J^*_m(x), m \in \Bbb N$} 
The concept of an \emph{edge-joint} between two simple undirected graphs $G$ and $H$ is defined below.
\begin{definition}
The edge-joint of two simple undirected graphs $G$ and $H$ is the graph obtained by linking the edge $vu,$ $v \in V(G), u \in V(H)$ and denoted, $G\rightsquigarrow_{vu}H.$
\end{definition}
\noindent Note: $G\rightsquigarrow_{vu}H = G \cup H + vu,$ $v \in V(G), u \in V(H).$\\ \\
The next theorem provides $Gut(J^*_n(x)\rightsquigarrow_{v_1u_1}J^*_m(x))$ in terms of $Gut(J^*_n(x))$ and $Gut(J^*_m(x))$. The edge-joint $J^*_n(x)\rightsquigarrow_{v_1u_1}J^*_m(x)$ is called \emph{trivial}. Edge-joints $J^*_n(x)\rightsquigarrow_{v_iu_j}J^*_m(x)$, $i\neq 1$ or $j\neq 1$ are called \emph{non-trivial}. For families (classes) of graphs such as paths $P_n$, cycles $C_n$, complete graphs $K_n$, Jaco graphs $J_n(f(x))$, etc, the notation is abbreviated as $P_n\rightsquigarrow_{vu}P_m = P^{\rightsquigarrow_{uv}}_{n,m}$ and $J^*_n(f(x))\rightsquigarrow_{v_iu_j}J^*_m(f(x)) = J^{\rightsquigarrow_{v_iu_j}}_{n,m},$ etc.
\begin{theorem}
For the underlying graphs $J^*_n(x)$ and $J^*_m(x)$ of the finitie Jaco Graphs $J_n(x), J_m(x),\\ n,m \in \Bbb N$ and $n \geq m \geq 2$:\\ \\
$Gut(J^*_n(x)\rightsquigarrow_{v_1u_1}J^*_m(x)) = Gut(J^{\rightsquigarrow_{v_1u_1}}_{n,m}) = Gut(J^*_n(x)) + Gut(J^*_m(x)) + \sum\limits_{\ell=2}^{n}d_{J^*_n(x)}(v_\ell)d_{J^*_n(x)}(v_1, v_\ell) +\\ \\ \sum\limits_{s=2}^{m}d_{J^*_m(x)}(u_s)d_{J^*_m(x)}(u_1, u_s) + \sum\limits_{t=2}^{m}(d_{J^*_n(x)}(v_1)+1)d_{J^*_m(x)}(u_t)(d_{J^*_m(x)}(u_1, u_t)+1) +\\ \\ \sum\limits_{k=2}^{n}\sum\limits_{t=2}^{m}d_{J^*_n(x)}(v_k)d_{J^*_m(x)}(u_t)(d_{J^*_n(x)}(v_1, v_k) + d_{J^*_m(x)}(u_1, u_t) + 1) + 4.$
\end{theorem}
\begin{proof}
Consider the underlying Jaco graphs, $J^*_n(x), J^*_m(x),$ with $n, m \in \Bbb N$ and $n \geq m \geq 2$ with $J_m(x)$ having prime Jaconian vertex $u_i$. Also label the vertices of $J^*_n(x)$ and $J^*_m(x)$; $v_1, v_2, v_3, ..., v_n$ and $u_1, u_2, u_3, ..., u_m$, respectively. Consider $J^{\rightsquigarrow_{v_1u_1}}_{n,m} = J^*_n(x) \cup J^*_m(x) + v_1u_1$. Without loss of generality  apply the piecewise definition:\\ \\
$Gut(J^{\rightsquigarrow_{v_1u_1}}_{n,m}) =\sum\limits_{k=1}^{n-1}\sum\limits_{\ell =k+1}^{n}d_{J^{\rightsquigarrow_{v_1u_1}}_{n,m}}(v_k) d_{J^{\rightsquigarrow_{v_1u_1}}_{n,m}}(v_\ell)d_{J^{\rightsquigarrow_{v_1u_1}}_{n,m}}(v_k, v_\ell) + \sum\limits_{t=1}^{m-1}\sum\limits_{s=t+1}^{m}d_{J^{\rightsquigarrow_{v_1u_1}}_{n,m}}(u_t) d_{J^{\rightsquigarrow_{v_1u_1}}_{n,m}}(u_s)d_{J^{\rightsquigarrow_{v_1u_1}}_{n,m}}(u_t, u_s) +\\ \\ 
\sum\limits_{k=1}^{n}\sum\limits_{t =2}^{m}d_{J^{\rightsquigarrow_{v_1u_1}}_{n,m}}(v_k) d_{J^{\rightsquigarrow_{v_1u_1}}_{n,m}}(u_t)d_{J^{\rightsquigarrow_{v_1u_1}}_{n,m}}(v_k, u_t) + d_{J^{\rightsquigarrow_{v_1u_1}}_{n,m}}(v_1) d_{J^{\rightsquigarrow_{v_1u_1}}_{n,m}}(u_1)d_{J^{\rightsquigarrow_{v_1u_1}}_{n,m}}(v_1, u_1).$\\ \\
\noindent Step 1(a): Consider vertex $v_1$ and vertex $v_\ell, 2 \leq \ell \leq n.$ By applying the definition of the Gutman index to this pair of vertices we have the term:\\ \\
$d_{J^{\rightsquigarrow_{v_1u_1}}_{n,m}}(v_1) d_{J^{\rightsquigarrow_{v_1u_1}}_{n,m}}(v_\ell)d_{J^{\rightsquigarrow_{v_1u_1}}_{n,m}}(v_1, v_\ell) = (d_{J^*_n(x)}(v_1) + 1) d_{J^*_n(x)}(v_\ell)d_{J^*_n(x)}(v_1, v_\ell) =\\\\ d_{J^*_n(x)}(v_1)d_{J^*_n(x)}(v_\ell)d_{J^*_n(x)}(v_1, v_\ell) + d_{J^*_n(x)}(v_\ell)d_{J^*_n(x)}(v_1, v_\ell).$\\ \\
By applying this step $\forall v_\ell, 2 \leq \ell \leq n$ we obtain:\\ \\
$\sum\limits_{\ell=2}^{n}d_{J^*_n(x)}(v_1)d_{J^*_n(x)}(v_\ell)d_{J^*_n(x)}(v_1, v_\ell) + \sum\limits_{\ell=2}^{n}d_{J^*_n(x)}(v_\ell)d_{J^*_n(x)}(v_1, v_\ell).$\\ \\
\noindent Step 1(b): For all ordered pairs of vertices $(v_k, v_\ell)_{k<\ell}$ with $2 \leq k \leq n-1$ and $3 \leq \ell \leq n$ we have that:\\ \\
$\sum\limits_{k=2}^{n-1}\sum\limits_{\ell=k+1}^{n}d_{J^{\rightsquigarrow_{v_1u_1}}_{n,m}}(v_k) d_{J^{\rightsquigarrow_{v_1u_1}}_{n,m}}(v_\ell)d_{J^{\rightsquigarrow_{v_1u_1}}_{n,m}}(v_k, v_\ell) = \sum\limits_{k=2}^{n-1}\sum\limits_{\ell=k+1}^{n}d_{J^*_n(x)}(v_k) d_{J^*_n(x)}(v_\ell)d_{J^*_n(x)}(v_k, v_\ell).$\\ \\ \\
By applying this step $\forall (v_k, v_\ell)_{k<\ell}$, $1 \leq k \leq n-1$ and $2 \leq \ell \leq n$, we obtain:\\ \\ 
$\sum\limits_{k=1}^{n-1}\sum\limits_{\ell=k+1}^{n}d_{J^*_n(x)}(v_k) d_{J^*_n(x)}(v_\ell)d_{J^*_n(x)}(v_k, v_\ell) + \sum\limits_{\ell=2}^{n}d_{J^*_n(x)}(v_\ell)d_{J^*_n(x)}(v_1, v_\ell) =\\\\ 
Gut(J^*_n(x)) + \sum\limits_{\ell=2}^{n}d_{J^*_n(x)}(v_\ell)d_{J^*_n(x)}(v_1, v_\ell).$\\ \\
\noindent Step 2: Similar to Step 1 we have that:\\ \\
$\sum\limits_{t=1}^{m-1}\sum\limits_{s =t+1}^{m}d_{J^{\rightsquigarrow_{v_1u_1}}_{n,m}}(u_t)d_{J^{\rightsquigarrow_{v_1u_1}}_{n,m}}(u_s)d_{J^{\rightsquigarrow_{v_1u_1}}_{n,m}}(u_t, u_s) =\sum\limits_{t=1}^{m-1}\sum\limits_{s=t+1}^{m}d_{J^*_m(x)}(u_t) d_{J^*_m(x)}(u_s)d_{J^*_m(x)}(u_t, u_s) +\\\\ \sum\limits_{s=2}^{m}d_{J^*_m(x)}(u_s)d_{J^*_m(x)}(u_1, u_s) = Gut(J^*_m(x)) + \sum\limits_{s=2}^{m}d_{J^*_m(x)}(u_s)d_{J^*_m(x)}(u_1, u_s).$\\ \\
\noindent Step 3: To conclude this step we will provide the next partial summation as a piecewise summation, to be:\\ \\
$\sum\limits_{k=1}^{n}\sum\limits_{t =2}^{m}d_{J^{\rightsquigarrow_{v_1u_1}}_{n,m}}(v_k) d_{J^{\rightsquigarrow_{v_1u_1}}_{n,m}}(u_t)d_{J^{\rightsquigarrow_{v_1u_1}}_{n,m}}(v_k, u_t) =\sum\limits_{t =2}^{m}d_{J^{\rightsquigarrow_{v_1u_1}}_{n,m}}(v_1) d_{J^{\rightsquigarrow_{v_1u_1}}_{n,m}}(u_t)d_{J^{\rightsquigarrow_{v_1u_1}}_{n,m}}(v_1, u_t) +\\ \sum\limits_{k=2}^{n}\sum\limits_{t =2}^{m}d_{J^{\rightsquigarrow_{v_1u_1}}_{n,m}}(v_k) d_{J^{\rightsquigarrow_{v_1u_1}}_{n,m}}(u_t)d_{J^{\rightsquigarrow_{v_1u_1}}_{n,m}}(v_k, u_t).$\\ \\ \\
\noindent Step 3(a): Consider vertex $v_1$ and vertex $u_t, 2 \leq t \leq m.$ By applying the definition of the Gutman index to this pair of vertices we have the term:\\ \\
$d_{J^{\rightsquigarrow_{v_1u_1}}_{n,m}}(v_1) d_{J^{\rightsquigarrow_{v_1u_1}}_{n,m}}(u_t)d_{J^{\rightsquigarrow_{v_1u_1}}_{n,m}}(v_1, u_t) =(d_{J^*_n(x)}(v_1) + 1)d_{J^*_m(x)}(u_t)(d_{J^*_m(x)}(u_1, u_t) +1).$\\ \\
By applying this step $\forall u_t, 2 \leq t \leq m$ we obtain:\\ \\
$\sum\limits_{t=2}^{m}(d_{J^*_n(x)}(v_1)+1)d_{J^*_m(x)}(u_t)(d_{J^*_m(x)}(u_1, u_t)+1).$\\ \\ \\
\noindent Step 3(b): Consider vertex $v_k, 2 \leq k \leq n$ and vertex $u_t, 2 \leq t \leq m.$ By applying the definition of the Gutman index to this pair of vertices we have the term:\\ \\
$d_{J^{\rightsquigarrow_{v_1u_1}}_{n,m}}(v_k) d_{J^{\rightsquigarrow_{v_1u_1}}_{n,m}}(u_t)d_{J^{\rightsquigarrow_{v_1u_1}}_{n,m}}(v_k, u_t) =d_{J^*_n(x)}(v_k)d_{J^*_m(x)}(u_t)(d_{J_n^*(x)}(v_1, v_k) + d_{J^*_m(x)}(u_1, u_t) +1).$\\ \\
By applying the step $\forall v_k, 2 \leq k \leq n$ and $\forall u_t, 2 \leq t \leq m$, we obtain:\\ \\ \\
$\sum\limits_{k+2}^{n}\sum\limits_{t=2}^{m}d_{J^*_n(x)}(v_k)d_{J^*_m(x)}(u_t)(d_{J_n^*(x)}(v_1, v_k) + d_{J^*_m(x)}(u_1, u_t) +1).$\\ \\ \\
\noindent Step 4: It is easy to see that:\\ \\
$d_{J^{\rightsquigarrow_{v_1u_1}}_{n,m}}(v_1) d_{J^{\rightsquigarrow_{v_1u_1}}_{n,m}}(u_1)d_{J^{\rightsquigarrow_{v_1u_1}}_{n,m}}(v_1, u_1) = 4.$\\ \\
\noindent \textbf{Final summation step:} Adding Steps 1 to 4 provides the result:\\ \\ $Gut(J^{\rightsquigarrow_{v_1u_1}}_{n,m}) = Gut(J^*_n(x)) + Gut(J^*_m(x)) + \sum\limits_{\ell=2}^{n}d_{J^*_n(x)}(v_\ell)d_{J^*_n(x)}(v_1, v_\ell) + \sum\limits_{s=2}^{m}d_{J^*_m(x)}(u_s)d_{J^*_m(x)}(u_1, u_s) +\\ \sum\limits_{t=2}^{m}(d_{J^*_n(x)}(v_1)+1)d_{J^*_m(x)}(u_t)(d_{J^*_m(x)}(u_1, u_t)+1) +\sum\limits_{k=2}^{n}\sum\limits_{t=2}^{m}d_{J^*_n(x)}(v_k)d_{J^*_m(x)}(u_t)(d_{J^*_n(x)}(v_1, v_k)+\\\\ d_{J^*_m(x)}(u_1, u_t) + 1) + 4.$
\end{proof}
\section{Conclusion} 
For the simple case $f(x) = x$ the calculation of the Gutman index for Jaco graph and the edge-joint between them is immediately complicated. Finding a result similar to Theorem 3.1 for $J^*_n(x)\rightsquigarrow_{v_iu_j}J^*_m(x)$, $i\neq 1$ or $j\neq 1$ \emph{(non-trivial edge-joints)} remains open. The single most important challege is to find a closed formula for the number of edges in $J_n(x)$. Such closed formula will enable finding a closed formula for distances between given vertices and a simplied formula for many invariants of Jaco graphs might result from such finding. Hence, important open questions remain such as: \emph{Is there a closed formula for the number of edges of} $J_n(x),$ $n \in \Bbb N$? \emph{Is there a closed formula for the cardinality of the Jaconian set} $\Bbb J(J_n(x))$ \emph{of} $J_n(x), n \in \Bbb N$? \emph{Is there a closed formula for} $d_{J^*_n(x)}(v_1, v_n)$ \emph{in} $J^*_n(x), n \in \Bbb N$?. Refer to [7] for further reading.\\ \\
\textbf{\emph{Open access:}} This paper is distributed under the terms of the Creative Commons Attribution License which permits any use, distribution and reproduction in any medium, provided the original author(s) and the source are credited. \\ \\
References (Limited) \\ \\
$[1]$  V. Andova, D. Dimitrov, J. Fink,  R. \v Skrekovski, \emph{Bounds on Gutman Index}, MATCH Communications in Mathematical and in Computer Chemistry, Vol 67 (2012), pp 515-524.\\
$[2]$ J.A. Bondy, U.S.R. Murty,\emph {Graph Theory with Applications,} Macmillan Press, London, (1976). \\
$[3]$  P. Dankelmann, I. Gutman, S. Mukwembi, H.C. Swart, \emph{The edge-Wiener index of a graph}, Discrete Mathematics, Vol 309 (2009), pp 3452-3457.\\
$[4]$ I. Gutman, \emph{Selected properties of the Schultz molecular topological index}, Journal of Chemical Information and Computer Sciences, Vol 34, (1994) pp 1087-1089.\\ 
$[5]$ J. Kok, P.  Fisher, B. Wilkens, M. Mabula, V. Mukungunugwa, \emph{Characteristics of Finite Jaco Graphs, $J_n(1), n \in \Bbb N$}, arXiv: 1404.0484v1 [math.CO], 2 April 2014. \\
$[6]$ J. Kok, P.  Fisher, B. Wilkens, M. Mabula, V. Mukungunugwa,  \emph{Characteristics of Jaco Graphs, $J_\infty(a), a \in \Bbb N$}, arXiv: 1404.1714v1 [math.CO], 7 April 2014.\\ 
$[7]$ J. Kok, C. Susanth, S.J. Kalayathankal, \emph{A Study on Linear Jaco Graphs}, arXiv:1506.06538v1, [math.CO], 22 June 2015, Journal of Informatics and Mathematical Sciences (Accepted).
\end{document}